\documentclass[11pt]{amsart}
\usepackage{amsmath}
\usepackage{minipage}
\usepackage{amsfonts}
\usepackage{amsthm}
\usepackage{amssymb}
\usepackage{mathrsfs}
\usepackage{latexsym}
\usepackage{verbatim}
\usepackage{a4wide}
\usepackage{enumerate}
\usepackage{hyperref}
\usepackage{graphics}
\usepackage{graphicx}

\renewcommand{\thetheoremName}

\def\sign{\mathrm{sign}}

\def\Z{\mathbf{Z}}

\def\R{\mathbf{R}}
\def\C{\mathbf{C}}

\def\eps{\epsilon}

\def\A{\mathcal{A}}

\DeclareMathOperator\im{im}

\newtheorem{theorem}{Theorem}[section]

\newtheorem{df}[theorem]{Definition}
\newtheorem{lemma}[theorem]{Lemma}

\newtheorem{conj}{Conjecture}

\newtheorem{remark}[theorem]{Remark}

\def\tareesidedbox#1{\setbox0=\hbox{$#1$}\dimen0=\wd0 \advance\dimen0 by3pt\rlap{\hbox{\vrule height9pt width.4pt depth2pt \kern-.4pt\vrule height9.4pt width\dimen0 depth-9pt\kern-.4pt \vrule height9pt width.4pt depth2pt}} \relax \hbox to\dimen0{\hss$#1$\hss}}
\def\ho#1{\tareesidedbox{#1}}

\newcommand{\UZ}{U \kern-.1em{Z}}

\newcommand{\VF}{V \kern-.07em{F}}

\def\Vol{\mathrm{Vol}}

\begin{document}

\title{Counting Perron numbers by absolute value}
\author{Frank Calegari and Zili Huang}
\thanks{The authors were supported in part by NSF  Grant
  DMS-1404620.}

\begin{abstract} We count various classes of algebraic integers of fixed degree by their largest absolute value.
The classes of integers considered include all algebraic integers,
Perron numbers, totally real integers, and totally complex integers. We give qualitative and quantitative results
concerning the distribution of Perron numbers, answering in part a question of W.~Thurston~\cite{Thurston}. 
\end{abstract}

\maketitle

\maketitle

\section{Introduction}

The goal of this paper is to address several questions concerning Perron numbers
suggested by a recent preprint of~W.~Thurston~\cite{Thurston}.
An algebraic integer $\alpha$ is a \emph{Perron number} if it has larger absolute value than
any of its Galois conjugates. How many Perron numbers are there? Although there
are numerous ways to order and count algebraic integers, in this context it seems
most natural to count (in  fixed degree) by absolute value, and this is what we do.
As well as counting Perron algebraic integers, we  count all algebraic integers.
Let $\A_N$ denote the set of algebraic integers of degree~$N$. 
For $\alpha \in \A_N$, let $\ho{\alpha}$ --- the \emph{house} of $\alpha$ --- denote the largest
absolute value of any conjugate~$\sigma \alpha$ of $\alpha$. An argument of
Kronecker shows that there are only finitely many elements of~$\A_N$ with house at most~$X$.
Specifically, any algebraic integer of degree $N$ all of whose conjugates' absolute values
is bounded by $X$ is the root of a polynomial whose coefficients are bounded strictly in terms of $X$ and $N$, and hence there
are only finitely many such $\alpha$.
Let $\A^{+}_N \subset \A_N$ denote the subset of totally real algebraic integers of degree~$N$.
Let $\A^{-}_{N} \subset \A_{N}$ denote the subset of algebraic integers of degree~$N$ which are totally
complex (this is empty unless~$N$ is even).  Let $\A^{P}_{N} \subset \A_N$ denote the algebraic integers of degree~$N$ which
are Galois conjugates to a Perron number.
Our first result gives an estimate for the sizes of these sets.

\begin{theorem} \label{thm:one} \label{theorem:count} As $X \rightarrow \infty$, 
$$\displaystyle{|\A^*_N(X)| = X^{\frac{N(N+1)}{2}} D^*_{N}}\left(1 + O\left(\frac{1}{X}\right)\right),$$
where $*$ is either unadorned, $P$, $+$, or $-$ when $N$ is even, and the 
constants $D^*_N$ are given as follows:
$$D_{N} = 
\displaystyle{ \prod_{k=0}^{m-1} \left(\frac{k!^2 2^{2k+1}}{(2k+1)!}\right)^2},  N = 2m; \quad
\ D_{N} = \displaystyle{\left(\frac{m!^2 2^{2m+1}}{(2m+1)!}\right) \prod_{k=0}^{m-1} \left(\frac{k!^2 2^{2k+1}}{(2k+1)!}\right)^2},
N = 2m+1,$$
$$
D^{P}_N = \frac{D_N}{N+1}, N = 2m; \quad  \quad  \quad
D^{P}_N = \frac{D_N}{N}, N = 2m+1,$$
$$D^{+}_N =  \prod_{k=0}^{N-1} \frac{2^{k+1} k!^2}{(2k+1)!}, \quad  \quad  \quad D^{-}_{2N} = \frac{2^{2N(N-1)} (2N)!}{N!^2} D^{+}_{2N}.$$
\end{theorem}
For example, given a ``random'' algebraic integer, the probability that
$\alpha$ is (conjugate to) a Perron algebraic integer is $1/N$ if $N$ is odd and $1/(N+1)$ if~$N$ is even.
Note that the last answer with the ratio $D_{2N}^-/D_{2N}^+$ proves Conjecture 5.2 in~\cite{Akiyama}.
This theorem reduces to understanding various integrals over the region
$\Omega_N$ in $\R^N$  which parametrize monic polynomials
 all of whose roots have absolute value at most~$1$.
If one imposes conditions on the signature, then one obtains corresponding
 regions $\Omega_{R,S} \subset \Omega_{N}$ with $N = R+2S$, the constants $D^{+}_N$ and $D^{-}_{2N}$ 
 are then the volumes of $\Omega_{N,0}$ and $\Omega_{0,N}$ respectively.
 There is a natural decomposition
$$\Omega_N = \coprod_{R+2S = N} \Omega_{R,S}.$$
Beyond counting algebraic integers in these classes, it is also of interest to try to understand what a ``typical''
such element is, under the constraint that the house of $\alpha$ is bounded by a fixed constant~$X$,
which leads us towards our next result.

\subsection{A question of Thurston} \label{section:thurston}
Thurston asked~\cite{Thurston}
whether one could understand the distribution of Perron numbers subject to the constraint that their absolute
values are bounded by a fixed real number~$X$. 
Recall that
a Perron algebraic integer is a real algebraic integer $\alpha$ with
$|\alpha| = \ho{\alpha}$ whose absolute is strictly larger than all its Galois conjugates. We say that a polynomial
with coefficients in $\R$ is Perron if it has a unique (necessarily real) largest (in terms of absolute value) root.
Usually one insists that a Perron number is a \emph{positive} real number, but with our definition $\alpha$
is Perron if and only if $- \alpha$ is also Perron. (The only change to the asymptotics is a factor of~$2$.)
We explain in section~\ref{section:perron} how to count Perron algebraic integers.
 In one experiment, Thurston attempted to model random polynomials
whose largest root is $\le 5$ by taking polynomials of degree~$21$ all of whose coefficients lie in the interval $[-5,5]$. The corresponding
roots showed a tendency to cluster the ratio of their absolute values to the largest root away from $|z| = 1$; we
include his figure here as figure~$1$.
\begin{figure}[!ht] \label{fig:one}
\begin{center}
  \includegraphics[width=60mm]{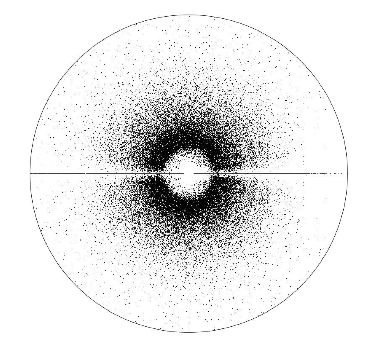}
  \end{center}
\caption{A plot from~\cite{Thurston} showing
the normalized  roots $\sigma \alpha/\alpha$ of the minimal polynomials for $5932$ degree~$21$ 
Perron numbers $\alpha$,
obtained by sampling 20,000 monic degree~$21$ polynomials
with integer coefficients in $[-5,5]$ and keeping those that have a root of absolute value at most~$5$ which is larger
than all other roots.}
\end{figure}
 However, we shall
explain why this picture is not accurate  representation of the entire space of
Perron polynomials.
As a point of comparison, 
Thurston sampled polynomials over a space with $11^{21}$ lattice points and volume $10^{21}$. 
On the other hand,  let  $\Omega^{P}_{21} \subset \Omega_{21}$ be the region
consisting of polynomials
with a unique largest root, and consider the scaled version of this space where the roots
are allowed to have absolute value at most~$5$. Then this region has volume
$$5^{210} \cdot D^{P}_{21} = \frac{2^{189} 5^{198}}{3^{24} 7^{10} 11^{11} 13^{9} 17^{5} 19^{3}}
\sim 8.308 \times 10^{143}.$$
Hence Thurston's samples are taken from a region which represents less than 
one $10^{123}$th of the entire space of Perron polynomials. As another illustration, the average value
 of $|a_{21}|$ over the correct region is 
 $$ 5^{21}\cdot \frac{88179}{524288}=\frac{3 \cdot 5^{21} \cdot 7 \cdot 13 \cdot 17 \cdot 19}{2^{19}} \sim 8.020 \times 10^{13},$$
which is not anywhere close to being in $[-5,5]$. Indeed, this value might \emph{a priori}
be considered surprisingly large,
given that the absoute maximum of the constant term $|a_{21}|$ is $5^{21} \sim 4.768 \times 10^{14}$. 
We should make clear that Thurston made no claims
that his experiment produced a faithful representation of $\Omega^{P}_{21}$,  and
he explicitly mentions the coefficients of a typical member
 of $\Omega^{P}_{21}$ appears to be ``much larger'' than~$5$.
Indeed, one
of the problems he posed is to formulate a good method for sampling ``randomly'' in this space. 
A natural approach to the latter question is to use a random walk Metropolis--Hastings algorithm. 
Figure~$2$, produced via such a random walk algorithm,
  is in agreement
with our theoretical results, such as Theorem~\ref{theorem:cluster} below.
The ``ring'' structure evident in Thurston's picture (of radius approximately~$1/5$)  is a consequence of the  fact that polynomials
with (suitably) small coefficients have roots which tend to cluster uniformly around the disc of radius one. This follows
in the radial direction by
a theorem of Erd\"{o}s and Tur\'{a}n~\cite{Erdos},  and for the absolute values from~\cite{Hughes}.
In contrast, the reality is that the conjugates of Perron polynomials will cluster around the boundary,
which is our second result:

\begin{theorem} \label{theorem:cluster} As $N \rightarrow \infty$, the  roots of a random polynomial in~$\Omega_N$
or~$\Omega^P_N$ are distributed uniformly about the unit circle. 
\end{theorem}

\begin{figure}[ht] \label{fig:metro} \centering \begin{minipage}[b]{0.45\linewidth}
\includegraphics[width=60mm]{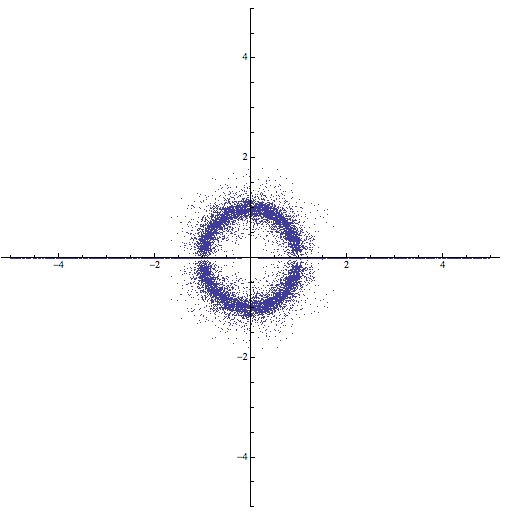}  
\end{minipage} \quad \begin{minipage}[b]{0.45\linewidth}
\includegraphics[width=60mm]{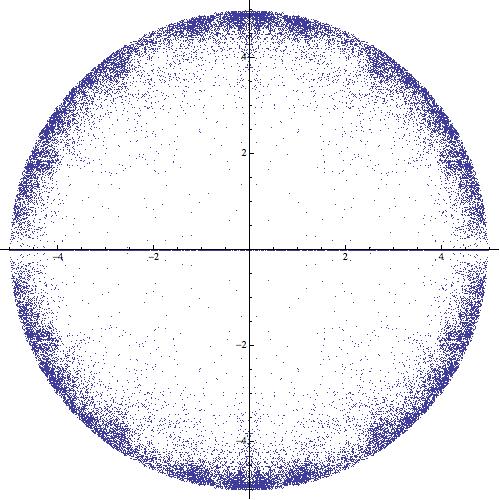} 
\end{minipage} 
\caption{The  first  graph consist of the roots of Perron polynomials with largest root $\le 5$ and
integral coefficients in~$[-5,5]$ as in~\cite{Thurston},  except no longer
normalizing by the largest (necessarily real) root~$\alpha$. The second graph consists of  roots of random Perron polynomials in~$\Omega^{P}_{21}$ scaled
to have a maximal root of absolute value $\le 5$ as generated by a random walk Metropolis--Hastings algorithm.
The second graph is in accordance with Theorem~\ref{theorem:cluster}.}
\end{figure}

\medskip

\subsection{Asymptotics} 
It is easy to give asymptotics for any product formula using Stirling's formula and its variants. For example, 
we have the following:

\begin{lemma}  \label{lemma:limit} 
 The probability
that a polynomial $P(x) \in \Omega_{N}$ has only real roots is
$$\frac{D^{+}_{N}}{D_{N}} \sim  
\frac{C  \cdot N^{1/8}}{2^{N^2/2}},$$
The probability that a polynomial $P(x) \in \Omega_{2N}$ has no real roots,
equivalently, that  $P(x) > 0$ for all $x$,  is
$$\frac{D^{-}_{2N}}{D_{2N}} \sim  \frac{2C}{\sqrt{2\pi} \cdot (2N)^{3/8}}$$
for the same constant~$C$ as above.
\end{lemma}

\begin{remark} \emph{With a little extra care, one can also identify the constant~$C$ above as
$$C = 2^{-1/24} e^{- 3/2 \cdot \zeta'(-1)} =  1.24514 \ldots.$$
}
\end{remark}

\begin{remark} \label{remark:Poonen} \emph{ A polynomial $P(x) \in \Omega_{2N}$ is positive everywhere if and only if it is positive on~$[-1,1]$.
On the other hand, for many classes of random models of polynomials, it is a theorem  of 
Dembo, Poonen, Shao, and~Zeitouni~\cite{Poonen} that a random polynomial whose coefficients are chosen
with (say) identical normal distributions with zero mean is positive in $[-\infty,\infty]$
 with probability $N^{-b + o(1)}$ and positive in $[-1,1]$ with probability $N^{-b/2 + o(1)}$
 for some universal constant~$b/2$, which 
 they estimate be $0.38 \pm  0.015$. On the other hand, 
the exponent occuring above is $3/8 = 0.375$. Is there any direct relationship between
these theorems? For example, does this suggest that~$b = 3/4$? 
}
\end{remark}

\subsection{The limit  \texorpdfstring{$N \rightarrow \infty$}{N->oo}}

As $N \rightarrow \infty$, we are still able to say something about the geometric spaces $\Omega_N$,
but the direct connection with algebraic integers becomes more tenuous. Given a fixed region
$\Omega$ with appropriate properties, it is quite reasonable to be able to count lattice points
in the large~$X$ limit as $\Omega$ is scaled appropriately. However, the error in any such estimate
will depend highly on $\Omega$, so this does not allow one to understand the lattice points in a sequence
$\Omega_N$ of regions simply in terms of the volume.  There are some known global constraints.
For example,
Kronecker proved that the only elements of $\A_N$  with house in $[0,1]$ are roots of unity,
and the only elements of $\A^{+}_N$ with house in $[-2,2]$ are of the form $\zeta + \zeta^{-1}$ for a root of unity $\zeta$.
This is consistent with our volume computations; the
 smallest
value of $X$ for which $\Vol(\Omega_N) X^{N(N+1)/2}$ is $\ge 1$ is 
$$1 + \frac{\log N}{N} + O \left(\frac{1}{N}\right),$$ 
whereas the corresponding value for
 $\Vol(\Omega_{N,0}) X^{N(N+1)/2}$ is
 $$2 + \frac{2 \log N}{N} + O \left(\frac{1}{N}\right).$$
 (In practice, there exist algebraic integers which are not roots of unity of house at least as small
 as $ 2^{1/N} \sim 1 + \log 2/N$.)

\subsection{Configuration Spaces} \label{section:realselberg}
A natural way to understand the spaces $\Omega_{R,S}$ is to consider the spaces defined by the roots.
In this way, one can relate integrals over $\Omega_{R,S}$ to integrals over nicer spaces at the
expense of including the factor coming from the Jacobian.
For example, consider the case of $\Omega_{N,0}$. There is a natural map:
$$[-1,1]^N \rightarrow \R^N$$
given by:
$$\phi: (x_1,\ldots,x_N) \rightarrow (s_1,\ldots,s_N),$$
where $s_m$ is the $m$th symmetric polynomial in the $x_i$s.  Suppose that
$$s_{m,k}:= \frac{\partial}{\partial x_k} s_m.$$
Then $s_{m,k}$ is the $(m-1)$th symmetric polynomial in the variables $x_i$ with $x_k$ omitted, and the
Jacobian matrix is given by 
$J(\phi_*) = [s_{m,k}]$. If $x_i = x_j$ then $s_{m,i} = s_{m,j}$ and the Jacobian vanishes. By comparing
degrees, it follows that  $|J(\phi_*)|$  is the  absolute value of the Vandermonde determinant.
Since $\phi$ is generically $N!$ to $1$,
it follows that
$$\int_{\Omega_{N,0}} dV = \frac{1}{N!} \int_{[-1,1]^N} \prod |x_i - x_j| dx_1 \ldots dx_N = D^{+}_N.$$
Yet the latter integral can be computed exactly because it is a special case of the integrals considered
by~\cite{Selberg}.  This is enough to prove the corresponding
claim in Theorem~\ref{thm:one} in this case.
Similar parameterizations allow us to write $\int_{\Omega_{R,S}}$ as a multiple
integral, but not all the integrals which arise have such nice product expressions.

\subsection{Selberg Integrals} 
 We now consider all monic polynomials with real coefficients whose roots have absolute
value at most one. We assume that the polynomial has degree $N$, and that the polynomial
has $R$ real roots and $S$ pairs of complex roots. Let $B(1)$ be the unit ball in $\C$.
There is a map
$$B(1)^S \times [-1,1]^R \rightarrow \Omega_{R,S} \subset \R^{N}$$
Given by
$$\phi:(z_1,\ldots,z_S,x_1,\ldots,x_R) \rightarrow (z_1,\ldots,z_S,\overline{z}_1,\ldots,\overline{z}_S,x_1,
\ldots,x_R) \rightarrow (s_1, \ldots, s_{N}),$$
where the $s_i$ are symmetric in the $N$ variables. 
The following is elementary:

\begin{lemma}
 The absolute value of the determinant  $|J(\phi_*)|$ is  the absolute value of
$$ 2^{S} \prod_{i=1}^{S} (z_i - \overline{z_i}) \prod_{i \ne j} (z_i - z_j)(z_i - \overline{z_j})
\prod_{i \ne j} (x_i - x_j) \prod_{i,j} (x_i - z_j)(x_i - \overline{z}_j)$$
\end{lemma}

The mapping from $B(1)^S \times [-1,1]^R$  to $\Omega_{R,S}$ is not $1$ to $1$.
 Rather, there is a generically faithful transitive  action of the group
 $\Z/2\Z \wr S_s \times S_r$ on the fibres.
Hence
$$\int_{\Omega_{R,S}} dV :=  
\frac{1}{R! S!}  \int_{[-1,1]^R} \int_{B(1)^S}
 \prod_{i=1}^{S} |z_i - \overline{z_i}| \prod_{i > j} |(z_i - z_j)(z_i - \overline{z_j})|^2
\prod_{i > j} |x_i - x_j| \prod_{i,j} |x - z_j|^2
  dV$$

\medskip

We shall discuss various integrals which are analogues of certain Selberg integrals.
These correspond to the contexts in which the roots are totally real, totally imaginary, or without any
restriction.

  \begin{df} Let $C_N(\alpha,T)$ denote the integral
  $$\int_{\Omega_N} |a_N|^{\alpha - 1} P(T) dV.$$
  \end{df}

  \begin{theorem} \label{theorem:main}
  There are equalities as follows. If $N = 2m+1$, then
  $$C_{N}(\alpha,T) = D_N \left(\prod_{k=0}^{m} \frac{1 + 2k}{\alpha + 2k}\right)
 \cdot \sum_{i=0}^{m}  \frac{2i+1}{2m+1}  \frac{m!^2}{(2m)!} \binom{2i}{i} \binom{2m-2i}{m-i} T^{2i+1},$$
 If $N = 2m$, then
  $$C_{N}(\alpha,T) = D_N  \cdot \left( \prod_{k=0}^{m-1} \frac{1 + 2k}{\alpha + 2k} \right) \cdot 
  \sum_{i=0}^{m} \frac{2i + \alpha}{2m + \alpha }  \frac{m!^2}{(2m)!} \binom{2i}{i} \binom{2m-2i}{m-i} T^{2i},$$
 \end{theorem}

This theorem is proved in~\S\ref{section:evaluate}.
By considering the leading term of the polynomial, this integral
formula implies that
$$
\int_{\Omega_{N}} |a_N|^{\alpha - 1} dV
 = D_N \left(\prod_{k=0}^{m} \frac{1 + 2k}{\alpha + 2k}\right) =  D_N \frac{(1/2)_{m+1}}{(\alpha/2)_{m+1}},$$
where $N = 2m$ or $2m+1$. Specializing further to $\alpha = 1$,
we deduce that
$\displaystyle{ \int_{\Omega_{N}}  dV = D_N}$. This latter integral was also computed by Fam~\cite{Fam}, the evaluation of $C_N(\alpha,T)$ above is similar.

\subsection{Comparison with classical Selberg Integrals}

In~\S\ref{section:subspace}, we define integrals $C^{+}_N$ and $C^{-}_N$ which are similar
to $C_N(\alpha,T)$ except the integral takes place over $\Omega_{N,0}$ or $\Omega_{0,N}$
respectively.
The integral $C_N(\alpha,T)$ is in some sense both the most natural (in that they are integrals over the entire
space $\Omega_N$) and unnatural, in that they are most naturally written as a sum of multiple integrals,
not each of which obviously admits an exact formula. The real Selberg integrals are closest to the classical
Selberg integrals, but even in this case they are not obvious specializations of known integrals.
To explain this futher, recall that Selberg's integral is a generalization of the  $\beta$-integral, which we write as
$$ \int_{-1}^{1} (1+t)^{\alpha- 1} (1-t)^{\beta-1} dt = \frac{2^{\alpha+\beta-1} \Gamma(\alpha) \Gamma(\beta)}{ 
\Gamma(\alpha+\beta)}.$$
Selberg's integral $\displaystyle{\int_{[0,1]^N} \prod t_i^{\alpha-1} (1-t_i)^{\beta -1}
\prod |t_i -t_j|^{2 \gamma} dt_1\ldots dt_N}$  can be written (up to easy factors) as
$$\int_{\Omega_{N,0}}  P(1)^{\alpha-1} P(-1)^{\beta-1} |\Delta_P|^{2\gamma-1},$$
where we integrate over the
 configuration space  of  monic polynomials  $P$ of degree~$N$ with real roots in $[-1,1]$, and $\Delta_P$ is the discriminant
 of the corresponding  polynomial.
On the other hand, when one  writes down similar integrals over $\Omega_{N}$, 
the corresponding integrals do not have nice product expressions. To take a simple example, one finds that
$$
\begin{aligned}
\frac{1}{3!} \int_{[-1,1]^3} \prod (x_i - x_j)^2 dx_1 dx_2 dx_3 
 & \ + \int_{0}^{2 \pi} \int_{0}^{1} \int_{-1}^{1} 4 r^3 \sin^2(\theta)(x^2 - 2 x r \cos(\theta) + r^2)^2 dx dr d\theta \\
 = &  \ \int_{\Omega_{3}} |\Delta_P|  = 
  \frac{32}{135} + \frac{41 \pi}{15} \end{aligned}
  $$
  The range of suitable integrals which have nice expressions over $\Omega_N$ seems to be
  fairly limited. Curiously enough, however, the integrals we do define do not
  even specialize to the $\beta$-integral over $\Omega_{1,0} = [-1,1]$, but rather to
  $$\int_{-1}^{1} |t|^{\alpha - 1}dt = \frac{2}{\alpha},$$
  and mild variants thereof. Moreover, there does not seem to be any flexibility in varying
  the power of the discriminant, as evidenced by the example over $\Omega_3$ above.
  On the other hand, one does have the following integral over~$\Omega_N$, (Theorem~\ref{theorem:selb}),
  which is the direct generalization of Selberg's integral when $\gamma = 1/2$:
  $$S_{N}(\alpha,\beta) =   \int_{\Omega_N}   P(1)^{\alpha-1} |P(-1)|^{\beta-1}  = 
$$
$$\prod_{k=1}^{2m} \frac{2^{\alpha + \beta + k - 2}}{\Gamma(\alpha + \beta + k - 1)}
\prod_{k=1}^{m} \Gamma(\alpha +\beta + k-1) \Gamma(k)
\prod_{k=1}^{m}  \Gamma(\alpha +k+1) \Gamma(\beta + k-1), \ \text{$N = 2m$ even},$$
$$\prod_{k=1}^{2m+1} \frac{2^{\alpha + \beta + k - 2}}{\Gamma(\alpha + \beta + k - 1)}
\prod_{k=1}^{m} \Gamma(\alpha +\beta + k-1) \Gamma(k)
\prod_{k=1}^{m+1}  \Gamma(\alpha +k+1) \Gamma(\beta + k-1), \ \text{$N = 2m+1$ odd}.$$

\subsection{Moments}
\label{section:moments}
The integral $C_N(\alpha)=\int_{\Omega_{N}} |a_N|^{\alpha-1} dV$   allows a precise description of the moments of $|P(0)| = |a_N|$ on $\Omega_N$, namely
$$E(\Omega_N,|a_N|^{\alpha - 1})
= \frac{\displaystyle{\int_{\Omega_{N}} |a_N|^{\alpha - 1} dV}}
{\displaystyle{\int_{\Omega_{N}} dV}} =  \prod_{k=0}^{\lfloor (N-1)/2 \rfloor} \frac{1 + 2k}{\alpha + 2k}.$$
Denote this function by $M_N(\alpha)$. 
Suppose that $\mu_N = H_N(x)dx$ is the distribution on $[0,1]$ of $|a_{N}|$ over
$\Omega_{N}$. Then  we know
$$M_N(\alpha):=\int^{1}_{0} x^{\alpha - 1} H_N(x) dx = \int^{\infty}_{0} e^{-\alpha t} H_N(e^{-t}) dt$$
 It follows that $H_N(e^{-t})$ is the inverse Laplace transform of
$M(\alpha)$.  Write $m = \lfloor (N-1)/2 \rfloor$. 
\begin{lemma}
The measure $\mu_N = H_N(x)dx$ is given  on $[0,1]$ by
$$\frac{(1 - x^2)^m}{\displaystyle{\left(\int^{1}_{0} (1 - x^2)^m dx\right)}} dx 
= \frac{2}{\sqrt{\pi}} \cdot \frac{\Gamma(m+3/2)}{\Gamma(m+1)} (1-x^2)^m dx.$$
\end{lemma}

By taking the logarithmic derivative with respect to $\alpha$ and
then sending $\alpha$ to $1$, we also find that, with the same~$m$, 
$$E(\Omega_N,\log |a_N|) 
= - \left(1 + \frac{1}{3} + \frac{1}{5} + \ldots + \frac{1}{2m+1}\right)
= - \frac{1}{2} \left( \log 2N + \gamma \right) + O\left(\frac{1}{N}\right),$$
and hence, for any root $\beta$,
$$E(\Omega_N,\log |\beta|) =- \frac{\log N}{2N} + O\left(\frac{1}{N}\right).$$
Since $|x| - 1 \ge \log |x|$ for $x \in B(1)$, we obtain the estimate:
$$E(\Omega_N,|\beta|) \ge 1 - \frac{\log N}{2N}.$$
It follows that the expected number of roots with absolute value less than $1 - \eps/2$ is less
than $\log|N| \eps$; in particular, it follows that almost all roots have absolute value
 within $1/N^{\delta}$
of~$1$ for any fixed $\delta < 1$ as $N \rightarrow \infty$.

\subsection{Perron Algebraic Integers}
\label{section:perron}
If $\alpha$ is a Perron algebraic integer, the necessarily $\alpha$ is real, since otherwise 
$|\alpha| = |\sigma \alpha|$. The converse is not true, however (take $\alpha = 2^{1/4}$). The
natural model for Perron numbers is polynomials whose largest root is real. (The subspace
of polynomials with more than one non-conjugate largest root has zero measure.)
Let $\Omega^{P}_{N}$ and $\Omega^P_{R,S}$ denote the corresponding spaces. 
Then there is a natural map:
$$\Omega_{N-1} \times [-1,1] \rightarrow \Omega^{P}_{N}$$
given by sending $P(x)$ to $t^{N-1} P(x/t)(x-t)$.
The effect on the variables is
$$b_k = t^{k}(a_{k} - a_{k-1}).$$
The Jacobian of this matrix is $(-1)^N \cdot t^{(N-1)(N+2)/2} P(1)$, and so, if $N=R+2S$,
$$\int_{\Omega^{P}_{R,S}} dV =    \int_{-1}^{1} |t|^{(N-1)(N+2)/2}   \int_{\Omega_{R-1,S}} |P(1)| dV
= \frac{4}{N(N+1)} \int_{\Omega_{R-1,S}} P(1) dV.$$
(Note, by assumption, that $P(1) \ge 0$, because all the (real) roots have absolute value $\le 1$, and the leading
coefficient is positive.) When $S = 0$, we see that
$$
\begin{aligned}
\int_{\Omega^{P}_{N,0}} dV = & \  
\frac{4}{N(N+1)} \int_{\Omega_{N-1,0}} P(1) dV \\
 =  & \  \frac{1}{N(N+1)}\frac{1}{(N-1)!} \int_{[-1,1]^N} \prod (1 - x_i) \prod |x_i - x_j| dx_1 \ldots dx_{N-1}. \end{aligned}$$
 The latter is another Selberg integral; in fact, by a special case of result of Aomoto~\cite{Aomoto},
 $$\frac{4}{N+1}  \int_{[-1,1]^{N-1}} \prod (1 - x_i) \prod |x_i - x_j| dx_1 \ldots dx_{N-1}
 = \int_{[-1,1]^{N}} \prod |x_i - x_j| dx_1 \ldots dx_{N},$$
 and hence
 $$\int_{\Omega^{P}_{N,0}} dV  = \frac{1}{N!}  \int_{[-1,1]^{N}} \prod |x_i - x_j| dx_1 \ldots dx_{N}
 = \int_{\Omega_{N,0}} dV.$$
 Of course, this is as expected, because a Perron polynomial with real roots is simply a polynomial
 with real roots.
 Let $D^{P}_{R,S} = \int_{\Omega^{P}_{R,S}} dV$.
Using the evaluation of our Selberg integral in Theorem~\ref{theorem:main}, we can compute the volume of $\Omega^P_{N}$.
Namely, we have:
\begin{lemma} \label{lemma:perronintegral}
There is an equality
$$\Vol(\Omega^P_{N}) = \frac{4}{N(N+1)} C_{N-1}(1,1) = \begin{cases}
\displaystyle{{D_N}/{N}}, 
& \text{$N$ odd}, \\
\displaystyle{{D_N}/{(N+1)}},
 & \text{$N$ even}. \end{cases}
$$ 
\end{lemma}

\begin{proof} The first equality follows from the computation above, the second
is an elementary manipulation with products of factorials. \end{proof}

We can do a similar analysis of expectation of $|a_{N}|^{\alpha -1}$ on $\Omega^{P}_N$
as we did with $\Omega_N$.
Namely
$$
\begin{aligned}
\int_{\Omega^{P}_{N}} |a_{N}|^{\alpha - 1}
dV =  & \ 
\int_{-1}^1 \int_{\Omega_{N-1}} |t|^{(N-1)(N+2)/2} |t|^{(\alpha-1)N} |a_{N-1}|^{\alpha - 1}  P(1) dV dt \\
=  & \   \int_{-1}^{1}|t|^{(N-1)(N+2)/2 + (\alpha - 1)N} dt   \int_{\Omega_{N-1}} |a_{N-1}|^{\alpha -1}  P(1) dV \\
= \frac{4 C_{N-1}(\alpha,1)}{N(N + 2 \alpha - 1)} 
\end{aligned}
$$
and hence
$$\displaystyle{E(\Omega^{P}_{N},|a_{N}|^{\alpha -1})
=  \begin{cases}
\displaystyle{\prod_{k=0}^m \frac{1+2k}{\alpha+2k}}, 
& \text{$N=2m+1$ odd}, \\
\displaystyle{ \frac{N+1}{N + 2 \alpha - 1} \prod_{k=0}^{m-1} \frac{1 + 2k}{\alpha + 2k}},
 & \text{$N=2m$ even}. \end{cases}}.$$
For example,
$E(\Omega^P_{21},|a_{21}|) = 
 \displaystyle{\frac{88179}{524288}}$, as mentioned earlier.

\begin{remark} \emph{
One could equally do this calculation insisting that Perron numbers be positive,
one simply replaces the integral over~$[-1,1]$ by an integral over~$[0,1]$. This makes
no difference to the proof of Theorem~\ref{theorem:cluster}
}
\end{remark}

\subsection{Selberg Integrals on subspaces} \label{section:subspace}

Let us consider the complex Selberg integral
$$C^{-}_N(\alpha) = \int_{\Omega_{0,N}} |a_{2N}|^{\alpha - 1}.$$
Recall that
$\displaystyle{(a)_k = a(a+1) \ldots (a+k-1) = \frac{\Gamma(a+k)}{\Gamma(a)}}$.
We have:

\begin{theorem} \label{theorem:complex}
Let $K_N = \displaystyle{\prod_{k=1}^{N-1}  \frac{k!^3 (k+1)! 2^{4k+1}}{(2k+1)!^2}}$. Then
$$C^{-}_N(\alpha)  = \int_{\Omega_{0,N}} |a_{2N}|^{\alpha - 1} = 
K_N \cdot \binom{2N}{N} \cdot
\frac{(4N+2 \alpha) (N + \alpha + 1/2)_{N}}{(N + \alpha)_{1+N}}  \prod_{k=1}^{N} \frac{(\alpha+k)_{2N+1-2k}}{(\alpha-1/2+k)_{2N+2-2k}},
$$
and
$C^{-}_N(1) = \displaystyle{\int_{\Omega_{0,N}} dV = D^{-}_{2N}}$.
\end{theorem}
We prove this theorem in~\S\ref{section:complex}.
It is possible to give explicit expressions for the integral
$C^{-}_N(\alpha,T) = \int_{\Omega_{0,N}} |a_{2N}|^{\alpha - 1} P(T)$ and
$C^{+}_N(\alpha,T) = \int_{\Omega_{N,0}} |a_N|^{\alpha - 1} P(T)$ as explicit polynomials
whose coefficients are various products of Pochhammer symbols. However, since we have
no use for them and the details are somewhat complicated, we omit them here.

 \subsection{The expected number of real roots}
 The expected number $r_N$ of real roots of a random polynomial in $\Omega_N$ is given, for small~$N$, as follows:
 
 $$0,1,\frac{2}{3}, \frac{17}{15}, \frac{32}{35}, \frac{43}{35}, \frac{1226}{1155}, \frac{1303}{1001},\
 \frac{10496}{9009},
 \frac{208433}{153153}, \frac{402}{323}, \frac{1367}{969}, \ldots $$

There is a  map:
$$R: \Omega_{N-1} \times [-1,1] \rightarrow \Omega_{N}$$
given by sending $P(x)$ to $P(x)(x-T)$. 
The Jacobian of this map is equal to $|P(T)|$. 

This map is not $1$ to $1$, rather, the image
in $\Omega_{R,S}$ has multiplicity~$R$. We exploit this fact as follows. Let $Z(P)$ denote
the number of real roots of~$P$. Then pulling back the measure under $R^*$, we deduce that
$$\int_{\Omega_{N}} Z(P) = \int_{-1}^{1} \int_{\Omega_{N-1}}  |P(T)|dV.$$
For $|T| < 1$, the inner integral strictly differs from $C_{N-1}(1,T)$, so we cannot
use our evaluation of $C_N(1,T)$ to compute this integral. On the other hand, 
if we compute a \emph{signed} integral, then we are computing  the probability that the number  of  real
roots
is odd. We thus deduce, as expected, that
$$ \frac{1}{D_{N+1}} \int_{-1}^{1} C_N(1,T) = \begin{cases} 1, & N \ \text{even}, \\
0, & N \ \text{odd}.\end{cases}$$

There is, however, the following conjectural formula for the integral of $|P(T)|$:

\begin{conj} \label{conj:absolute} If $N = 2m$, then, for $T \in [-1,1]$,
{\footnotesize
$$\frac{1}{D_N} \int_{\Omega_N} |P(T)| = \frac{1}{2^{2m} \binom{2m}{m}} \left( \sum_{k=0}^{m} \frac{2m-2k+1}{2m+1} 
\binom{2m-2k}{m-k} \binom{2k}{k} T^{2k} \right)
 \left( \sum_{k=0}^{m} 
\binom{2m-2k}{m-k} \binom{2k}{k} T^{2k} \right).$$
}
If $N = 2m+1$, then, for $T \in [-1,1]$,
{\footnotesize
$$\frac{1}{D_N} \int_{\Omega_N} |P(T)| =  \frac{1}{2^{2m+2} \binom{2m}{m}}  \left( \sum_{k=0}^{m} \frac{2m-2k+1}{2m+1} 
\binom{2m-2k}{m-k} \binom{2k}{k} T^{2k} \right)
 \left( \sum_{k=0}^{m+1} 
\binom{2m+2-2k}{m+1-k} \binom{2k}{k} T^{2k} \right).$$
}
\end{conj}

Using this, we can say quite a bit about the explicit expectations $r_{N}$ appearing above.
There does not seem to be a closed form for $r_N$, but there is a nice relation between
$r_{2N+1}$ and $r_{2N}$, which comes down to an identity which may be proved
using Zeilberger's algorithm.
\begin{lemma} \label{lemma:zeil} Assume Conjecture~\ref{conj:absolute}.
There is an equality:
$$(2N+1 - r_{2N+1}) = \frac{4N+3}{4N+1} (2N - r_{2N}),$$
or
$$r_{2N+1} = \left(\frac{3+4N}{1+4N}\right) r_{2N} + \frac{1}{4N+1}.$$
\end{lemma}

\medskip

We now turn to asymptotic estimates of $r_N$. The exact arguments depend
on the parity, but they are basically the same in either case (they are also equivalent by  Lemma~\ref{lemma:zeil} above). Let us assume that
$N-1=2m$.  Then 
{\footnotesize
$$r_N  = \frac{1}{D_N} \int_{\Omega_N} Z(P) = 
 \frac{D_{N-1}}{D_N}
 \sum \sum  \frac{1}{2^{2m} \binom{2m}{m}}
 \frac{2(2i+1)}{(2m+1)(4m-2i-2j+1)} 
\binom{2i}{i} \binom{2j}{j}  \binom{2m-2i}{m-i} 
\binom{2m-2j}{m-j},$$
}
which we can write as
$$\frac{(2m+1)! (2m)!}{2^{4m+1} m!^4}   \sum \sum  \frac{2(2i+1)}{(2m+1)(4m-2i-2j+1)} 
 \frac{
 \displaystyle{
 \binom{2i}{i} \binom{2j}{j}  \binom{2m-2i}{m-i} 
\binom{2m-2j}{m-j}
}}{
\displaystyle{\binom{2m}{m}^2}}
$$
Asymptotically,
$$\frac{(2m+1)! (2m)!}{2^{4m+1} m!^4} = \frac{1}{\pi} + \frac{m^{-1}}{4 \pi}  + \ldots $$
Without this term, we can write the rest of the sum as:
$$s_N:= \sum \sum \frac{2(2m-2i+1)}{2m+1} \frac{1}{2i+2j+1}   \binom{2i}{i} \binom{2j}{j} 2^{-2i-2j} 
\prod_{k=m-i+1}^{m} \left(1 +  \frac{1}{2k-1}  \right)  \prod_{k=m-j+1}^{m} \left(1 + \frac{1}{2k-1}  \right)$$
Let us estimate the contribution to this sum coming from terms where $i + j < A$ and $A = \eps \cdot m$. This
will constitute the main term.
For a lower bound, note that the final product is clearly at least one, and that $i$ is at most~$A$. Hence
the sum is certainly bounded below by
$$\sum_{k<A} \sum_{i+j=k} 2 \left(1 - \frac{2A}{2m+1} \right) \frac{1}{2k+1} \binom{2i}{i}   \binom{2j}{j} 2^{-2i-2j}.$$
Since
 $$\sum_{i+j = k} \binom{2i}{i}   \binom{2j}{j} 2^{-2i-2j} = 1,$$
the lower bound is at least
$$\sum_{k<A}  2 \left(1 - \frac{2A}{2m+1} \right) \frac{1}{2k+1} \ge (1 - \eps) \left( \log(m) + \log(\eps) \right) + O(1).$$
The~$O(1)$ factor depends on~$\eps$, and should be thought
of as an estimate as~$m \rightarrow \infty$ with~$\eps$ fixed.
On the other hand, we have an upper bound for the sum given by noting that
$$\sum_{k=m-A+1}^{m} \log \left( 1 + \frac{1}{2k-1}  \right)
 \le  \sum_{k=m-A+1}^{m} \frac{1}{2k-1} = - \frac{1}{2}\log(1 - \eps) + o(1),$$
 leading to an upper bound for the sum above of the form:
 $$\sum_{k<A}  \frac{2}{2k+1} \cdot \frac{1}{1 - \eps}  \le  \left(\frac{ \log(m) + \log(\eps)}{1 - \eps}  \right) + O(1).$$
We now give an upper bound for the remaining sum. The initial factor is certainly at most~$2$, and $(2i + 2j + 1)^{-1} \le (2A + 1)^{-1}\le 1/(2A)$.
Hence an upper bound is given by
$$\sum_{k\geq A} \sum_{i+j=k} \frac{1}{A}  \binom{2i}{i}   \binom{2j}{j} 2^{-2i-2j}
\prod_{k=m-i+1}^{m} \left(1 +  \frac{1}{2k-1}  \right)  \prod_{k=m-j+1}^{m} \left(1 + \frac{1}{2k-1}  \right).$$
Note that we include terms in this sum with~$k<A$ --- this only increases
the sum, but restores some symmetry.
In particular,
the sum
is symmetric in $i \mapsto m-i$ and $j \mapsto m-j$, so after introducing a factor of~$4$, we may assume
that~$i$ and~$j$ are both at most~$m/2$.
In this range, we have the upper estimate
$$\sum_{k=m-i+1}^{m} \log \left( 1 + \frac{1}{2k-1}  \right) 
 \le  \sum_{k=m-\frac{m}{2}+1}^{m} \frac{1}{2k-1} \le \frac{\log(2)}{2} + o(1).$$
 Hence our remaining sum is bounded above by
 $$\sum_{k} \sum_{i+j=k} \frac{4}{A}  \binom{2i}{i}   \binom{2j}{j} 2^{-2i-2j} \cdot 2
 \le \frac{8m}{A} = \frac{8}{\eps}.$$
Combining our two estimates, we find that, for~$\eps$ fixed and~$m \rightarrow \infty$,
$$ 1 - \eps + o(1) \le \frac{s_N}{\log m} \le \frac{1}{1 - \eps} + o(1),$$
and thus~$s_N \sim \log m \sim \log N$. Note that in order to obtain a better estimate (the second term, for example), we would have
to be more careful about the dependence of the error terms on~$\eps$. 
 The analysis for $N = 2m+1$ is very similar.
Hence, since $r_N \sim (1/\pi) \cdot s_N$,
 we deduce the following:
 
 \begin{theorem} Assume Conjecture~\ref{conj:absolute}.  If $r_N$ denotes
  the expected number
 of zeros of a polynomial in $\Omega_N$, then, as $N \rightarrow \infty$,
 $$r_N \sim \frac{1}{\pi} \log N.$$
 \end{theorem}

\begin{remark} \emph{Note that, in the Kac~model of random polynomials (where the coefficients
are independent normal variables with mean~$0$), the expected number of
zeros of a polynomial of degree~$N$ is $2/\pi \log N$, and the expected
number in the interval~$[-1,1]$ is  $1/\pi \log N$~\cite{Kac}}
\end{remark}

Note that, for any $[a,b] \subset [-1,1]$, the integral 
$$\frac{1}{D_N} \int_{a}^{b} \int_{\Omega_N} |P(T)|$$
gives the expected number of zeros in the interval $[a,b]$. When $[a,b]$ does
not contain either~$1$ or $-1$, the answer is particularly simple as $N \rightarrow \infty$.

\begin{theorem}  Assume Conjecture~\ref{conj:absolute}.
For a fixed interval $[a,b]$ with $-1 < a < b < 1$, the expected number
of zeros in $[a,b]$ of a polynomial in $\Omega_N$ equals, as $N \rightarrow \infty$,
$$ \frac{1}{\pi} \int^{b}_{a} \frac{1}{1 - T^2} = \frac{1}{2 \pi}
\log \left|\frac{(1-a)(1+b)}{(1+a)(1-b)}\right|.$$
\end{theorem}

\begin{proof} The argument is similar (but easier) to the asymptotic computation
of~$r_N$ above. We consider the case $N-1=2m$, the other case is similar.
In some fixed interval $[a,b]$ not containing $\pm 1$,  the integrand converges  uniformly to
$$ \frac{1}{\pi}\left(\sum_{k=0}^{\infty} \frac{1}{4^k} \binom{2k}{k} T^{2k} \right)^2
= \frac{1}{\pi} \left(\frac{1}{\sqrt{1-T^2}}\right)^2,
$$
from which the result follows.
\end{proof}

It is interesting to note that this is exactly the density function of real zeros
for random power series~\cite{Shepp}. This is also the limit density for Kac 
polynomials, which converge to random power series. 
This suggests that random polynomials in~$\Omega_N$, might, in some sense
to be made precise, approximate random power series for large~$N$. This
lends some credence to the spectulations in~\ref{remark:Poonen}.

\section{Counting Polynomials}
\label{section:count}

Once we have formulas for volumes, the corresponding count of polynomials is fairly
elementary.

\begin{lemma} Let $\Omega \subset \R^N$ be a closed compact region such that
$\partial \Omega$ is contained in a finite union of algebraic varieties of codimension~$1$.  Let $\Omega(X) \subset \R^N$ denote image of $\Omega$ under
the stretch $[X,X^2,\ldots,X^N]$. Then the number of lattice points in the interior $\Omega(X)$ is
$$\Vol(\Omega) X^{N(N+1)/2} \left(1 + O \left(\frac{1}{X} \right) \right).$$
\end{lemma}

\begin{proof} This essentially follows from ``Davenport's lemma" in~\cite{Dav1} and~\cite{Dav2}. With its conditions being met by our hypothesis, Davenport's result tells us that the main term is the volume of $\Omega(X)$, or $\Vol(\Omega) X^{N(N+1)/2}$. The error term should be $O(\max\lbrace V_d(\Omega(X))\rbrace)$, where $\max\lbrace V_d(\Omega(X))\rbrace$ denotes the greatest $d$-dimensional volume of any of its projections onto a $d$-dimensional coordinate hyperplane, $\forall d\in \lbrace1,2,\cdots,n-1\rbrace$. In our case, as $X$ is large, this largest projection is clearly the one onto the last $N-1$ coordinates, which has volume proportional to $X^2 \cdot X^3 \cdots X^N$. Therefore the error term is $O\left(X^{\frac{N(N+1)}{2}-1}\right)$.
\end{proof}

From this, it is easy to deduce:

\begin{lemma} The number of irreducible polynomials of signature $(R,S)$ all of
whose roots have absolute value at most~$X$ is
$$\Vol(\Omega_{R,S}) X^{N(N+1)/2} \left(1 + O \left(\frac{1}{X} \right) \right).$$
The number of irreducible Perron polynomials of signature $(R,S)$ all of
whose roots have absolute value at most~$X$ is
$$\Vol(\Omega^P_{R,S}) X^{N(N+1)/2} \left(1 + O \left(\frac{1}{X} \right) \right).$$
\end{lemma}

\begin{proof} The first claim without the assumption of irreducibility follows from the previous lemma.
If a polynomial is reducible, however, then it factors as a product of two polynomials
each of which is monic (by Gauss) and has roots less than $X$ by assumption. Hence the
number of reducible factors is
$$\ll \sum_{A+B = N}^{A,B > 0} \Vol(\Omega_A) \Vol(\Omega_B) X^{A(A-1)/2 + B(B-1)/2}
= O\left(X^{\frac{N(N+1)}{2} - (N-1)}\right).$$
The argument for Perron polynomials is the same.
\end{proof}

Combining this with the four relevant integrals (that for~$D_N$ coming from
the remarks following Theorem~\ref{theorem:main}, for~$D^{P}_N$ from
Lemma~\ref{lemma:perronintegral}, for $D^{+}_N$ by the remarks in~\S\ref{section:realselberg},
and for $D^{-}_{2N}$ from Theorem~\ref{theorem:complex}), this proves Theorem~\ref{theorem:count}.

\section{Evaluating Integrals}
\label{section:evaluate}

We begin by evaluating the integral in Theorem~\ref{theorem:main}.
The integral $D_N = \displaystyle\int_{\Omega_N} dV$ was first computed by Fam in~\cite{Fam}.
His method also allows one to easily compute~$C_N(\alpha,T)$.
Let $\Omega_N(a_N)$ denote the intersection of~$\Omega_N$ with the
hyperplane where the last coefficient is fixed. Then Fam shows that
$\Omega_N(a_N)$ maps to $\Omega_{N-1}$ under a very explicit linear
transformation. Explicitly, 
$$\Omega_N(a_N)
 = T_{N-1} {\Omega}_{N-1},$$
where $T_{N-1}$ is the following matrix:
$$T_{N-1} = I_{N-1} + a_{N} \widehat{I}_{N-1},$$
for the identity matrix $I_{N-1}$ and the anti-diagonal matrix ($1$s on the anti-diagonal and $0$s
elsewhere)
$\widehat{I}_{N-1}$.
Explicitly, this takes the polynomial $Q(T) \in \Omega_{N-1}$ to  the polynomial
$$P(T) = T Q(T) +  a_N Q(1/T) T^{N-1} \in \Omega_{N}(a_N)$$
Recall that
$$C_N(\alpha,T) = \int_{\Omega_{N}} |a_N|^{\alpha - 1} P(T) dV.$$
Then we have
$$C_N(\alpha,T) = \int_{-1}^{1} \int_{\Omega_N(a_N)} |a_N|^{\alpha - 1} P(T)  dV d a_N.$$
Applying the change of coordinates above, we deduce that
$$C_N(\alpha,T) = \int_{-1}^{1} \int_{\Omega_{N-1}} |a_N|^{\alpha - 1} (T Q(T) +  a_N Q(1/T) T^{N-1})
 |\det(T_{N-1})| dV d a_N.$$
 By exchanging the order of integration and computing the determinant of~$T_{N-1}$, it follows that:
$$C_{N}(\alpha,T) =  T A_N C_{N-1}(1,T) + T^{N-1}  B_N C_{N-1}(1,1/T),$$
 where:
 $$A_{N} = \int_{-1}^{1} |t|^{\alpha - 1} (1 - t^2)^{(N-1)/2}dt, \  \text{$N$ odd},
\int_{-1}^{1} |t|^{\alpha - 1} (1 - t^2)^{(N-2)/2})( 1 + t)dt, \  \text{$N$ even},$$
and
$$B_{N}
 =  \int_{-1}^{1} |t|^{\alpha - 1} t (1 - t^2)^{(N-1)/2}dt, \  \text{$N$ odd},
 \int_{-1}^{1} |t|^{\alpha - 1} t (1 - t^2)^{(N-2)/2})( 1 + t)dt, \  \text{$N$ even},$$
Both of these integrals are special cases of Euler's beta integral and may be evaluated
explicitly. Theorem~\ref{theorem:main} follows easily by induction.
The evaluation of the integral
$S_N(\alpha,\beta):=\displaystyle{\int_{\Omega_N} P(1)^{\alpha - 1} P(-1)^{\beta - 1}}$ is similar.
The key point is that, under the substitution above with $P(T) = T Q(T) + a_N Q(1/T) T^{N-1} $, we have
$$P(1) = Q(1) + a_N Q(1), \qquad P(-1) = - Q(-1) + (-1)^{N-1} a_N Q(-1).$$
Recall that
$$\int_{-1}^{1} (1 -t)^{a-1} (1+t)^{b-1} dt = 2^{a+b-1} \int_{0}^{1} (1-x)^{a-1} x^{b-1} dx
= 2^{a+b-1} \frac{\Gamma(a) \Gamma(b)}{\Gamma(a+b)}.$$
Suppose that~$N=2m$ is even. Then
$$
\begin{aligned}
 \int_{-1}^{1} (1 + t)^{\alpha - 1}
(1 +  (-1)^N t)^{\beta - 1} \vert\det(T_{N-1})\vert dt =  & \ 
 \int_{-1}^{1} (1 + t)^{\alpha - 1}
(1 +  t)^{\beta - 1} (1+t)(1 - t^2)^{m-1} dt\\
=  \int_{-1}^{1} (1 + t)^{\alpha + \beta + m - 2}
(1 -  t)^{m - 1}  dt
= & \ 2^{\alpha + \beta + 2m - 2} \frac{\Gamma(\alpha +\beta + m-1) \Gamma(m)}{\Gamma(\alpha + \beta + 2m - 1)}. \end{aligned}
$$
Suppose that~$N=2m+1$ is odd. Then
$$ \begin{aligned}  \int_{-1}^{1} (1 + t)^{\alpha - 1}
(1 +  (-1)^N t)^{\beta - 1} \vert\det(T_{N-1})\vert dt = & \ 
 \int_{-1}^{1} (1 + t)^{\alpha - 1}
(1 -  t)^{\beta - 1} (1 - t^2)^{m} dt \\
=  \int_{-1}^{1} (1 + t)^{\alpha +m - 1}
(1 -  t)^{\beta + m - 1}  dt
= & \ 2^{\alpha + \beta + 2m - 1} \frac{\Gamma(\alpha +m) \Gamma(\beta + m)}{\Gamma(\alpha + \beta + 2m)}. \end{aligned}
$$
By induction, we deduce the following:
\begin{theorem} \label{theorem:selb} There is an equality:
$$\begin{aligned}
S_{N}(\alpha,\beta) = & \  \prod_{2k \le N}2^{\alpha + \beta + 2k - 2} \frac{\Gamma(\alpha +\beta + k-1) \Gamma(k)}{\Gamma(\alpha + \beta + 2k - 1)}
\prod_{2k+1 \le N} 2^{\alpha + \beta + 2k - 1} \frac{\Gamma(\alpha +k) \Gamma(\beta + k)}{\Gamma(\alpha + \beta + 2k)}\\
= & \  \prod_{k=1}^{N} \frac{2^{\alpha + \beta + k - 2}}{\Gamma(\alpha + \beta + k - 1)}
\prod_{2k \le N} \Gamma(\alpha +\beta + k-1) \Gamma(k)
\prod_{2k-1 \le N} \Gamma(\alpha +k-1) \Gamma(\beta + k-1).\end{aligned}$$
\end{theorem}

\subsection{The integral \texorpdfstring{$C^{-}_{N}(\alpha)$}{CNA}}
\label{section:complex}

\begin{lemma} \label{lemma:smallint}
For non-negative integers $p_i$ and $q_i$, we have
  $$\int_{B(1)^S}   \prod_{i=1}^{S} 
|z_i|^{2(\alpha - 1)} z^{p_i-1}_i
 \overline{z}^{q_i-1}_i  
    \mathrm{sign}(\im(z_i)) = (\sqrt{-1})^S \cdot 
\begin{cases}
\displaystyle{\prod_{i=1}^{S} \frac{4}{2\alpha + p_i + q_i-2}
\cdot  \frac{1}{p_i - q_i}},
& \text{$p_i - q_i$ all odd}, \\
\\
0, & \text{else}. \end{cases}$$
\end{lemma}

\begin{proof} The complex integral decomposes as a product over each $B(1)$. Changing to polar
coordinates, we
have
$$\begin{aligned} \int_{B(1)} z^{(\alpha - 1) + (p-1)}  \overline{z}^{(\alpha - 1) + (q-1)} \sign(\im(z)) 
= & \ \left(\int_{0}^{ \pi}  - \int_{\pi}^{2 \pi} \right) \int_{0}^{1} r^{2 \alpha + p + q -3} e^{i(p-q) \theta} d r d \theta
\\
= & \ \frac{4}{2 \alpha + p + q -2} \cdot \frac{\sqrt{-1}}{p-q}, \end{aligned}$$
if $p-q$ is odd and is trivial otherwise, from which the result follows.
\end{proof}

For a complex number $z$, note that $|z - \overline{z}| = \sqrt{-1} \cdot \sign(\im(z)) ( \overline{z}-z)$.
Thus the absolute value of the Vandermonde of signature $(R,S)$ has the following expansion:
$$\begin{aligned}
\prod_{i > j} &  |x_i - x_j| \prod (x_i - z_j)(x_i - \overline{z}_j) 
\prod_{i > j}  |z_i -z_j|^2 |z_i - \overline{z}_j|^2  \prod |z_i - \overline{z_i}|  \\
& \ = 
(\sqrt{-1})^S \sum_{S_N} \sign(\sigma) \prod_{i=1}^{R} x^{\sigma(i)-1}_i  \prod_{j=1}^{S} z^{\sigma(R+2j-1)-1}_j 
\overline{z}^{\sigma(R+2j)-1}_j \sign(\im(z_j)) 
\end{aligned}
$$

Expanding out the integrand for $C^{-}_N(\alpha)$
term by term and integrating using Lemma~\ref{lemma:smallint} and then combining the two
powers of $(\sqrt{-1})^N$, we find
$$C^{-}_{N}(\alpha) = \frac{1}{N!} \int_{\Omega_{0,N}} |a_N|^{\alpha - 1} dV$$
$$ = \frac{(-1)^N}{N!} \sum_{S_{2N}} \sign(\sigma) \prod_{j=1}^{N} \frac{4}{(2 \alpha + \sigma(2j-1) + \sigma(2j) -2)}
\frac{1}{\sigma(2j-1) - \sigma(2j)},$$
Here the sum is over permutations $\sigma$ such that $\sigma(2j)$ and $\sigma(2j-1)$ are neither
 both odd nor even. Given any such $\sigma$, swapping the values of
 $\sigma(2j)$ and $\sigma(2j-1)$ changes the sign \emph{both}
  $\sign(\sigma)$ and
 $\sigma(2j-1) - \sigma(2j)$, and leaves everything else unchanged. Hence we may sum only
 over $\sigma$ such that $\sigma(2j)$ is even, after including an extra factor of~$2^N$ (the order of the stabilizer of the set of~$N$ unordered pairs). The set of
 permutations such that $\sigma(2j)$ is even is simply $S_N \times S_N$. The ``diagonal''
 copy of this group permutes the factors of the product. Hence we are reduced to the sum
$$\frac{2^N N! (-1)^N}{N!} \sum_{S_{N} \times 1} \sign(\sigma) \prod_{j=1}^{N} \frac{4}{(2 \alpha + \sigma(2j-1) + \sigma(2j) -2)}
\frac{1}{\sigma(2j-1) - \sigma(2j)},$$
where now $\sigma$ fixes the odd integers. Absorbing the $2^N$  and $(-1)^N$ factors into the product, we arrive at
the sum
$$ \sum_{S_{N} \times 1} \sign(\sigma) \prod_{j=1}^{N} \frac{8}{(2 \alpha +2j-1 + \sigma(2j) -2)}
\frac{1}{\sigma(2j) - 2j+1}.$$
Yet this we may recognize as simply the Leibnitz expansion of the determinant of the following matrix:
$$M_N =  \left[\frac{8}{(2 \alpha + 2i + 2j-3)(2i - 2j 
+ 1)}\right].$$
Hence $C^{-}_{N}(\alpha) =  \det M_N$. This determinant is simply a Cauchy determinant. 
Let us write $x_{i}$ for $2i = 2,\ldots,2N$ and $y_{j}$ for $2j-1 = 1,\ldots, 2N-1$, so then
$$M_N =  \left[\frac{8}{(2 \alpha + x_{i} + y_j-2)(x_i - y_j)}\right] = \left[\frac{1}{X_i - Y_j}\right],$$
where
$$X_i= \frac{x^2_i + 2 \alpha x_i - 2 x_i}{8} , \qquad Y_j =  
\frac{y^2_j + 2 \alpha y_j - 2 y_j}{8}.$$
Note that $X_i - X_j$ and $Y_i - Y_j$ factor into linear factors. By Cauchy's determinant formula,
we may compute this determinant as a product of linear forms, and specializing as above
we recover our formula for $C^{-}_N(\alpha)$ (after a certain amount of simplification with
products of factorials).

\subsection{The distribution of roots in \texorpdfstring{$\Omega_N$}{ON}.}
In this section, we prove that the roots of almost all polynomials in~$\Omega_N$
distribute nicely around the unit circle. We have already shown that the absolute values
are close to one in section~\ref{section:moments}, so to establish the radial symmetry
we could apply the results of Erd\"{o}s and Tur\'{a}n~\cite{Erdos}. In fact, it is most convenient
to prove both radial and absolute value distribution simultaneously using the nice formulation
of~\cite{Hughes}. The main point of that paper is to show that the quantity
$$F_N = \log \left( \sum_{i=0}^{N} |a_i| \right) - \frac{1}{2} \log |a_0| - \frac{1}{2} \log |a_N|$$
governs the behavior in both cases, and that it suffices to show that $F_N$ is $o(N)$. In fact,
we shall be able to obtain a much more precise estimate (logarithmic in~$N$) which could
be used to give quite refined qualitative bounds if so desired. Since $\log |a_0| = 0$ and $\log |a_N|$
was estimated in section~\ref{section:moments}, the main task is to obtain bounds on the $a_i$.
It is most natural to estimate integrals of the quantities $|a_i|^2$ over~$\Omega_N$, and this is what we do.
We begin with a preliminary lemma.

\begin{lemma} \label{lemma:eq}
If $N$ is even, then 
$$ \frac{D_{N-1}}{D_N}   \int_{-1}^{1} a_N  |\det(T_{N-1})|  d a_{N}
=   \frac{D_{N-1}}{D_N}   \int_{-1}^{1} a^2_N  |\det(T_{N-1})|  d a_{N} = \frac{1}{N+1}.$$
If $N$ is odd, then
$$ \frac{D_{N-1}}{D_N}   \int_{-1}^{1} a_N  |\det(T_{N-1})|  d a_{N} = 0, \   \frac{D_{N-1}}{D_N}   \int_{-1}^{1} a^2_N  |\det(T_{N-1})|  d a_{N} = \frac{1}{N+2}.$$
\end{lemma}

Let 
$$A_N(i) =  \frac{1}{D_N} \int_{\Omega_N}  a_i dV.$$
Note that $A_N(i) = 0$ unless~$i$ is even.
On the other hand, $A_N(i)$ 
is determined exactly by the coefficients of $C_N(1,T)$, namely, if $N = 2m+1$, or $N = 2m$, then
$$A_N(2m-2i) =
   \frac{2i+1}{2m+1} \frac{m!^2}{(2m)!}  \binom{2i}{i}
\binom{2m-2i}{m-i}.$$
In either case, one has the (easy) inequality $|A_N(i)| \le 1$.
We now consider the integrals:
$$A_N(i,j) = \frac{1}{D_N} \int_{\Omega_N}  a_i a_j dV.$$
If $i,j < N$, then:
$$
\begin{aligned}
D_N \cdot A_N(i,j) =  & \ \int_{\Omega_{N-1}} \int_{-1}^{1} 
(a_i + a_{N-i} a_N)(a_j + a_{N-j} a_N)   |\det(T_{N-1})|  d a_{N} dV \\
= & \  \int_{\Omega_{N-1}}  a_i a_j dV \int_{-1}^{1}  |\det(T_{N-1})|  d a_{N}  \\
 & \ +   \int_{\Omega_{N-1}}  (a_i a_{N-j} + a_j a_{N-i}) dV  \int_{-1}^{1}  a_N |\det(T_{N-1})|  d a_{N}  \\
 & \ +   \int_{\Omega_{N-1}}  a_{N-i} a_{N-j} dV \int_{-1}^{1}  a^2_N  |\det(T_{N-1})|  d a_{N}  \\
  \end{aligned}
 $$
 If $i < N$ and $j = N$, then:
 $$
\begin{aligned}
D_N \cdot A_N(i,N) =  & \ \int_{\Omega_{N-1}} \int_{-1}^{1} 
(a_i + a_{N-i} a_N) a_N   |\det(T_{N-1})|  d a_{N} dV \\
= & \  \int_{\Omega_{N-1}}  a_i  dV \int_{-1}^{1}  a_N  |\det(T_{N-1})|  d a_{N}  \\
 & \ +   \int_{\Omega_{N-1}}  a_{N-i} dV  \int_{-1}^{1} a^2_N |\det(T_{N-1})|  d a_{N}   \end{aligned}
 $$
 If $i = j = N$, then $A_N(i,j)$ is  $1/(N+1)$ if~$N$ is even and $1/(N+2)$ if~$N$ is odd.
 
 \begin{lemma} \label{lemma:over} The inequality $|A_N(i,j)| \le N^3$ holds for all~$i,j,N$.
 \end{lemma}
 
 \begin{proof} The result is certainly true for~$N = 1$.
 We proceed by induction. Assume that neither~$i$ nor~$j$ is equal to~$N$.
 By the recurrence relation above, the triangle inequality, and Lemma~\ref{lemma:eq}, we
 deduce that
$$\begin{aligned}
|A_N(i,j)| \le  & \ |A_{N-1}(i,j)| + \frac{1}{N+1}
\left( |A_{N-1}(i,N-j)| + |A_{N-1}(j,N-i)|\right) 
\\ & \ +
\frac{1}{N+1} |A_{N-1}(N-i,N-j)|  \le 
 (N-1)^3 \left(1 + \frac{3}{N+1} \right)  < N^3. \end{aligned}$$
 Something similar (but easier) occurs when either~$i$ or~$j$ is~$N$ (using 
 the  inequality $|A_N(i)| \le 1$ noted above).
\end{proof}

\begin{remark} \emph{It seems  from some light calculation that the inequality
$|A_N(i,j)| \le 2$ (or at least $O(1)$) may hold for all~$N$, 
 although we have not tried very hard to prove this,
because the inequality above completely suffices for our purposes --- indeed all that matters
is that the bound is sub-exponential. 
}
\end{remark}

\begin{lemma} \label{lemma:soup}
The following inequality holds for all~$N$:
$$\frac{1}{D_N} \int_{\Omega_N} \sum_{i=0}^{N} |a_i| \le (N+1)^2 N^3.$$
\end{lemma}

\begin{proof}
By Cauchy--Schwartz,
$$\sum_{i=0}^{N} |a_i| \le (N+1) \sqrt{\sum_{i=0}^{N} |a_i|^2}
\le (N+1) \sum_{i=0}^{N} |a_i|^2,$$
where we use the fact that $\sqrt{x} \le x$ if $x \ge 1$ (note that $a_0 = 1$). It follows that
$$\frac{1}{D_N} \int_{\Omega_N} \sum_{i=0}^{N} |a_i|  \le (N+1) \sum_{i=0}^{N} |A_N(i,i)|,$$
and the result follows from Lemma~\ref{lemma:over}.
\end{proof}

\label{section:proofcluster}
We now prove Theorem~\ref{theorem:cluster}, which we restate now:

\begin{theorem} As $N \rightarrow \infty$, the  roots of a random polynomial in~$\Omega_N$
or~$\Omega^P_N$ are distributed uniformly about the unit circle.
\end{theorem}

\begin{proof} We first consider $\Omega_N$. Following~\cite{Hughes}, consider the quantity
$$F_N = \log \left(\sum_{i=0}^{N} |a_i| \right) - \frac{1}{2} \log |a_0| - \frac{1}{2}   \log |a_N|.$$
Note that $a_0 = 1$, so $\log |a_0| = 0$. This also implies that the first term in this sum is non-negative.
On the other hand, certainly $|a_N| \le 1$, so the last term is also non-negative, and $F_N \ge 0$ for
every point of~$\Omega_N$. We have already computed that
$$-E(\Omega_N,\log |a_N|) = \frac{1}{2}\log N + O(1);$$
a similar result holds for~$\Omega^P_N$ by the computation at the end of~\S\ref{section:perron}. 
Let $\Omega_N(100)$ denote the region where the first term in the above expression
for~$F_N$ has absolute value at least
$100 \log N$.  (The constant $100$ is somewhat arbitrary, it is relevant only that
$100 > 5 + 1$.) By Lemma~\ref{lemma:soup} we have
$$(N+1)^2 N^3 \ge \frac{1}{D_N} \int_{\Omega_N} \sum_{i=0}^{N} |a_i|
\ge \frac{1}{D_N} \int_{\Omega_N(100)}N^{100} 
= \frac{\Vol(\Omega_N(100))}{\Vol(\Omega_N)} N^{100}.$$
It follows that the part of $\Omega_N$ where $F_N$ is not between $0$ and $100 \log |N|$
is a vanishingly small part  of $\Omega_N$ for~$N$ large (by a large power of~$N$).
The same is true for~$\Omega^P_N$, because the volume of this latter space is (roughly)~$1/N$
times the volume of~$\Omega_N$. The result then follows from Theorem~1 of~\cite{Hughes}.
\end{proof}

\bibliographystyle{amsalpha}
\bibliography{Perronsmall}

 \end{document}